\documentclass[12pt]{amsart}
\usepackage{amssymb,latexsym}
\input epsf
\usepackage{graphicx}
\newtheorem{theorem}{Theorem}[section]
\newtheorem{proposition}[theorem]{Proposition}
\newtheorem{corollary}[theorem]{Corollary}
\newtheorem{definition}[theorem]{Definition}

\newtheorem{lemma}[theorem]{Lemma}
\newtheorem{example}[theorem]{Example}

\newtheorem{note}[theorem]{Note}

\begin{document}
\title[Homology of non-crossing
partition lattices] {Geometrically constructed bases for homology
of non-crossing partition lattices}
\author[Kenny]{Aisling~Kenny}
\address{School of Mathematical Sciences\\
Dublin City University\\
Glasnevin, Dublin 9\\
Ireland} \email{aisling.kenny9@mail.dcu.ie}

\begin{abstract}
For any finite, real reflection group $W$, we construct a
geometric basis for the homology of the corresponding non-crossing
partition lattice. We relate this to the basis for the homology of
the corresponding intersection lattice introduced by Bj\"{o}rner
and Wachs in \cite{BW} using a general construction of a generic
affine hyperplane for the central hyperplane arrangement defined
by $W$.
\end{abstract}

\maketitle

\section{Introduction}

Let $W$ be a finite, real reflection group acting effectively on
$\mathbf{R}^n$. In \cite{BW} Bj\"{o}rner and Wachs construct a
geometric basis for the homology of the intersection lattice
associated to $W$. There is another lattice associated to $W$
called the non-crossing partition lattice. In \cite{Shell},
Athanasiadis, Brady and Watt prove that the non-crossing partition
lattice is shellable for any finite Coxeter group $W$. Zoque
constructs a basis for the top homology of the non-crossing
partition lattice for the $A_n$ case in \cite{Zoque} where the
basis elements are in bijection with binary trees.\\

A geometric model $X(c)$ of the non-crossing partition lattice is
constructed in \cite{Lattices}. In this paper, we use $X(c)$ to
construct a geometric basis for the homology of the non-crossing
partition lattice that corresponds to $W$. We construct the basis
by defining a homotopy equivalence between the proper part of the
non-crossing partition lattice and the $(n-2)$-skeleton of $X(c)$.
We exhibit an explicit embedding of the homology of the
non-crossing partition lattice in the homology of the intersection
lattice, using the general construction of a generic affine
hyperplane $H_\mathbf{v}$.

\section{Preliminaries}

We refer the reader to \cite{Bourbaki} and \cite{Humphreys} for
standard facts and notation about finite reflection groups. As in
\cite{Lattices} we fix a fundamental chamber $C$ for the
$W$-action with inward unit normals $\alpha_1,\dots,\alpha_n$ and
let $r_1,\dots,r_n$ be the corresponding reflections. We order the
inward normals so that for some $s$ with $1\leq s\leq n$, the sets
$\{\alpha_1,\dots,\alpha_s\}$ and
$\{\alpha_{s+1},\dots,\alpha_n\}$ are orthonormal. We fix a
Coxeter element $c$ for $W$ where $c=r_1r_2\dots r_n$. As in
\cite{Lattices} we define a total order on roots by
$\rho_i=r_1\dots r_{i-1}\alpha_i$ where the $\alpha$'s and $r$'s
are defined cyclically mod $n$. The positive roots relative to the
fundamental chamber are $\{\rho_1,\rho_2,\dots,\rho_{nh/2}\}$
where $h$ is the order of $c$ in $W$ \cite{Stein}. Let $T$ denote
the reflection set of $W$. This consists of the set of reflections
$r(\rho_i)$ where $\rho_i$ is a positive root and $r(\rho_i)$ is
the reflection in the hyperplane orthogonal to $\rho_i$. For $w\in
W$, let $\ell(w)$ denote the smallest $k$ such that $w$ can be
written as a product of $k$ reflections from $T$. The partial
order $\preceq$ on $W$ is defined by declaring for $u,w\in W$:
\begin{eqnarray} \label{order}
  u\preceq w &\Leftrightarrow& \ell(w)=\ell(u) + \ell(u^{-1}w).
\end{eqnarray}
The subposet of elements of $W$ that precede $c$ in the partial
order \eqref{order} is denoted $NCP_c$. The subposet $NCP_c$ forms
a lattice (by \cite{Lattices} for example),
and is called the non-crossing partition lattice.\\

We now review the definition of the geometric model $X(c)$ of
$NCP_c$ constructed in \cite{Lattices}. The spherical simplicial
complex $X(c)$ has as vertex set the set of positive roots
$\{\rho_1,\rho_2,\dots,\rho_{nh/2}\}$. An edge joins $\rho_i$ to
$\rho_j$ if $i<j$ and $r(\rho_j)r(\rho_i)$ is a length $2$ element
preceding $c$. The vertices $\langle \rho_{i_1}, \dots,
\rho_{i_{k}}\rangle$ form a $(k-1)$-simplex if they are pairwise
joined by edges. For each $w \preceq c$, $X(w)$ is defined to be
the subcomplex of $X(c)$ consisting of those simplices whose
vertices have the property that the corresponding
reflections precede $w$.\\

Finally, we recall some notation and standard facts about posets
(\cite{Tools}, \cite{Quillen}). Let $P$ denote a bounded poset
with minimal element $\hat{0}$ and maximal element $\hat{1}$. The
proper part of the poset $P$ is denoted by $\bar{P}$ and defined
to be $\bar{P}=P\setminus \{\hat{0}, \hat{1}\}$. Let $|P|$ denote
the simplicial complex associated to $P$, that is the simplicial
complex whose vertices are the elements of the poset $P$ and whose
simplices are the non-empty finite chains in $P$. We say that the
poset $P$ is contractible if the simplicial complex $|P|$ is
contractible. For $\Delta$ a simplicial complex, let $P(\Delta)$
denote the poset of simplices in $\Delta$ ordered by inclusion.
The barycentric subdivision of the simplicial complex $\Delta$ is
the simplicial complex $|P(\Delta)|$ and is denoted $sd(\Delta)$.

\section{Homotopy Equivalence}\label{homotopy}

We begin with the observation that every simplex in $X(c)$ defines
a non-crossing partition. Recall from Lemma $4.8$ of
\cite{Lattices} that if $\{\tau_1,\dots,\tau_k\}$ is the ordered
vertex set of a simplex $\sigma$ of $X(c)$ then
\[\ell(r(\tau_1)\dots r(\tau_k)c)=n-k.\]
In particular, $r(\tau_k)\dots r(\tau_1)$ is a non-crossing
partition of length $k$.

\begin{definition} We define $f:P(X(c))\rightarrow
NCP_c$ by
\[f(\sigma)=r(\tau_k)\dots r(\tau_1)\] where $\sigma$
is the simplex of $X(c)$ with ordered vertex set
$\{\tau_1,\dots,\tau_k\}$.
\end{definition}

\begin{lemma}\label{lemma}
The map $f$ is a poset map.
\end{lemma}

\begin{proof}
Let $\sigma=\{\tau_1,\dots,\tau_k\} \in P(X(c))$ and let
$\theta \preceq \sigma$. Therefore,
$\theta=\{\tau_{i_1},\dots,\tau_{i_l}\}$ for some $1\leq
i_1<\dots< i_l\leq k$. Note that for any roots $\rho$ and $\tau$,
we have $r(\rho)r(\tau)=r(\tau)r(\rho')$, where
$\rho'=r(\tau)[\rho]$. We can use this equality to conjugate the
reflections in $f(\theta)$ to the beginning of the expression for
$f(\sigma)$. Therefore $f(\theta) = r(\tau_{i_l})\dots
r(\tau_{i_1}) \preceq r(\tau_{k})\dots r(\tau_{1})=f(\sigma)$.
\end{proof}

By definition of $f$, $f^{-1}(c)$ is the set of maximal elements
in $P(X(c))$ and $f^{-1}(e)$ is empty. We therefore can consider
the induced map,
\[\hat{f}:\hat{P}(X(c))\rightarrow \overline{NCP_c}\]
where $\hat{P}(X(c))$ is the poset obtained from $P(X(c))$ by
removing the maximal elements. Note that $\hat{P}(X(c))$ is the
poset of simplices of the $(n-2)$-skeleton of $X(c)$.\\

The following result was proved by Athanasiadis and Tzanaki in
Theorem $4.2$ of \cite{Christos} in the more general setting of
generalised cluster complexes and generalised non-crossing
partitions. However, we include the proof of the specific case
here.

\begin{theorem} \label{thrm}The map $\hat{f}$ is a homotopy equivalence.
\end{theorem}

\begin{proof}
Since $f$ is a poset map by lemma \ref{lemma},
$\hat{f}:\hat{P}(X(c))\rightarrow \overline{NCP_c}$ is a poset
map. We intend to apply Quillen's Fibre Lemma \cite{Quillen} to
this map $\hat{f}$. Following the notation of \cite{Quillen}, we
define the subposet $\hat{f}_{\preceq w}$ of $\hat{P}(X(c))$ for
$w \in \overline{NCP_c}$ by
\[\hat{f}_{\preceq w}=\{\sigma \in \hat{P}(X(c)) : \hat{f}(\sigma)\preceq w\}.\]

We claim that $\hat{f}_{\preceq w}=P(X(w))$. Assuming the claim,
the theorem follows from Proposition $1.6$ of \cite{Quillen} if
$|P(X(w))|$ is contractible. It is shown in Corollary $7.7$ of
\cite{Lattices} that $X(w)$ is contractible for all $w \in NCP_c$.
Since $X(w)$ and $sd(X(w))$ are homeomorphic (by \cite{Quillen}
for example) and
$|P(X(w))|=sd(X(w))$, it follows that $|P(X(w))|$ is contractible.\\

To prove the claim we first show that $\hat{f}_{\preceq
w}\subseteq P(X(w))$. If $\sigma \in \hat{f}_{\preceq w}$, then
$e\prec \hat{f}(\sigma)\preceq w \prec c$ by definition of
$\hat{f}_{\preceq w}$. By applying lemma \ref{lemma} to the
reflections corresponding to vertices of $\sigma$, it follows that
$\sigma \in P(X(w))$. To show that $P(X(w))\subseteq
\hat{f}_{\preceq w}$, let $\sigma\in P(X(w))$. If $\sigma$ has
ordered vertex set $\{\tau_1,\dots,\tau_k\}$, then
$r(\tau_i)\preceq w$ for each $i$ by definition of $X(w)$. Then
$\hat{f}(\sigma)=r(\tau_k)\dots r(\tau_1)\preceq c$. By equation
$3.4$ of \cite{Lattices}, we know that since
$\hat{f}(\sigma)\preceq c$, $w\preceq c$ and each
$r(\tau_i)\preceq w$ then $\hat{f}(\sigma)=r(\tau_k)\dots
r(\tau_1)\preceq w$. Therefore, $\sigma \in \hat{f}_{\preceq w}$.
\end{proof}

\begin{corollary}
$|\overline{NCP_c}|$ has the homotopy type of a wedge of spheres,
one for each facet of $X(c)$.
\end{corollary}

\begin{proof}
The map $\hat{f}$ induces a homotopy equivalence
$|\hat{f}|:|\hat{P}(X(c))|\rightarrow |\overline{NCP_c}|$. The
simplicial complex $X(c)$ is a spherical complex that is convex
and contractible (Theorem $7.6$ of \cite{Lattices}). Let $Y$
denote the subspace of $X(c)$ obtained by removing a point from
the interior of each facet. Then $|\hat{P}(X(c))|$ is a
deformation retract of $Y$ and therefore has the homotopy type of
a wedge of $(n-2)$ spheres. The number of such spheres is equal to
the number of facets of $X(c)$.
\end{proof}

\begin{note} This is a more direct proof of the result in
corollary \mbox{4.4} of \cite{Shell} where it is proved that for a
crystallographic root system, the M\"{o}bius number of $NCP_c$ is
equal to $(-1)^n$ times the number of maximal simplices of
$|NCP_c|$, which can also be viewed as positive clusters
corresponding to the root system.\end{note}

\section{Homology Embedding}\label{homology}

We now briefly review the results in \cite{BW} where geometric
bases for the homology of intersection lattices are constructed.
Let $\mathcal{A}$ be a central and essential hyperplane
arrangement in $\mathbf{R}^n$. We refer to the connected
components of $\mathbf{R}^n\setminus\mathcal{A}$ as regions. We
let $L_{\mathcal{A}}$ denote the set of intersections of
subfamilies of $\mathcal{A}$, partially ordered by reverse
inclusion. We refer to $L_{\mathcal{A}}$ as the intersection
lattice of $\mathcal{A}$.\\

Homology generators are found by using a non-zero vector
$\mathbf{v}$ such that the hyperplane $H_\mathbf{v}$, which is
through $\mathbf{v}$ and normal to $\mathbf{v}$, is generic. This
means that $\mbox{dim}(H_\mathbf{v}\cap X)=\mbox{dim}(X)-1$ for
all $X\in L_\mathcal{A}$. In Theorem $4.2$ of \cite{BW}, it is
proven that the collection of cycles $g_R$ corresponding to
regions $R$ such that $R\cap H$ is nonempty and bounded, form a
basis of $\tilde{H}_{d-2}(\bar{L}_\mathcal{A})$ where $H$ is an
affine hyperplane, generic with respect to $\mathcal{A}$. Lemma
$4.3$ of \cite{BW} states that for each region $R$, the affine
slice $R\cap H_{\mathbf{v}}$ is nonempty and bounded if and only
if $\mathbf{v}\cdot\mathbf{x}>0$ for all $\mathbf{x}\in R$. At
this point, we refer the reader to figure \ref{fig} which
illustrates this basis for $W=C_3$. The figure shows the
stereographic projection of the open hemisphere satisfying
$\mathbf{v}\cdot\mathbf{x}>0$ and is combinatorially equivalent to
the projection onto $H_{\mathbf{v}}$. Each region in the figure
which is non-empty and bounded contributes a generator to the
basis for the homology of the intersection lattice.\\

The fact that the hyperplane $H_\mathbf{v}$ is generic is
equivalent to the fact that $0\notin H_\mathbf{v}$ and $H\cap X
\neq \emptyset$ for all $1$-dimensional subspaces $X\in
L_\mathcal{A}$ (Section $4$ of \cite{BW}). We will refer to a
non-zero vector in a one dimensional subspace $X\in L_\mathcal{A}$
as a ray. It is therefore sufficient to check that $H_\mathbf{v}$
is generic with respect to the set of rays. In section
\ref{vector}, we describe for any $W$, the general construction of
a vector $\mathbf{v}$ with $H_\mathbf{v}$ generic. In section
\ref{special}, we use the construction of $\mathbf{v}$ to
explicitly embed the homology of the non-crossing partition
lattice in the homology of the intersection lattice.

\subsection{Construction of a generic vector for general finite reflection groups}
\label{vector} Let $\{\tau_1,\dots,\tau_n\}$ be an arbitrary set
of linearly independent roots. Since the number of roots is finite
and rays occur at the intersection of hyperplanes, it follows that
the number of unit rays is finite. Hence, the set
$\{\mathbf{r}\cdot \rho \mid \mathbf{r} \mbox{ a unit ray, } \rho
\mbox{ a root}
 \}$ is finite and
\[\lambda=\mbox{min}\{|\mathbf{r}\cdot\rho| :  \mathbf{r} \mbox{ a unit ray, }
\rho \mbox{ a root and } \mathbf{r}\cdot\rho\neq 0 \}\] is a well
defined, positive, real number.  It will be convenient to use the
auxiliary quantity $a=1+1/\lambda$.

\begin{proposition}\label{vprop}
Let $\mathbf{v}=\tau_1+a\tau_2+a^2\tau_3+\dots+a^{n-1}\tau_n$ and
$\mathbf{r}$ be a unit length ray. Then
$|\mathbf{r}\cdot\mathbf{v}|\geq \lambda$. In particular,
$H_\mathbf{v}$ is generic.
\end{proposition}

\begin{proof}
Let $\mathbf{r}$ denote a unit length ray. Since
$\{\tau_1,\dots,\tau_n\}$ is a linearly independent set,
$\mathbf{r}\cdot\tau_k\neq 0$ for some $\tau_k$. Let $k$ be the
index with $1\leq k\leq n$ satisfying
\[\mathbf{r}\cdot\tau_k\neq 0, \mbox{ and } \mathbf{r}\cdot\tau_{k+1}= 0,\dots, \mathbf{r}\cdot\tau_n=0.\]
By replacing $\mathbf{r}$ by $-\mathbf{r}$ if necessary, we can
assume that $\mathbf{r}\cdot\tau_k>0$ and hence
$\mathbf{r}\cdot\tau_k\geq \lambda$ by the definition of
$\lambda$. We now compute $\mathbf{r}\cdot\mathbf{v}$.
\begin{eqnarray*}
  \mathbf{r}\cdot\mathbf{v} &=& \mathbf{r}\cdot(\tau_1+a\tau_2+a^2\tau_3+\dots+a^{n-1}\tau_n) \\
   &=& \mathbf{r}\cdot\tau_1+a(\mathbf{r}\cdot\tau_2)+a^2(\mathbf{r}\cdot\tau_3)+\dots+a^{n-1}(\mathbf{r}\cdot\tau_n) \\
   &=& \mathbf{r}\cdot\tau_1+a(\mathbf{r}\cdot\tau_2)+a^2(\mathbf{r}\cdot\tau_3)+\dots+a^{k-1}(\mathbf{r}\cdot\tau_k)+0\\
   &\geq& -1+a(-1)+a^2(-1)+\dots +a^{k-2}(-1)+a^{k-1}(\lambda) \\
   &=& -1(1+a+a^2+\dots +a^{k-2})+a^{k-1}(\lambda) \\
&=&\lambda.
\end{eqnarray*}
The last equality follows from the formula for the sum of a
geometric series and the fact that $\lambda=1/(a-1)$.
\end{proof}

\subsection{Specialising the generic hyperplane}
\label{special} In order to relate the homology basis for
non-crossing partition lattices to the homology basis for the
corresponding intersection
lattice, we apply the operator \\
$\mu=2(I-c)^{-1}$ from \cite{Lattices} to $X(c)$ to obtain a
complex which we will call $\mu(X(c))$ and which is the positive
part of the complex $\mu(AX(c))$ studied in \cite{PtoA}. The
complex $\mu(X(c))$ has vertices $\mu(\rho_1),\dots,
\mu(\rho_{nh/2})$ and a simplex on $\mu(\rho_{i_1}), \dots,
\mu(\rho_{i_k})$ if \[\rho_1\leq \rho_{i_1}<\dots<\rho_{i_k}\leq
\rho_{nh/2} \mbox{ and } \ell(r(\rho_{i_1})\dots
r(\rho_{i_k})c)=n-k.\] The walls of the facets of $\mu(AX(c))$ are
hyperplanes. Since regions considered in \cite{BW} are bounded by
reflection hyperplanes, this provides the connection between the
two and explains why we use $\mu(X(c))$ instead of
$X(c)$.\\

We now apply Proposition \ref{vprop} to the case where
$\tau_1,\dots,\tau_n$ are the last $n$ positive roots. Thus we set
$\tau_i=\rho_{nh/2-n+i}$. Since $\{\tau_1,\dots,\tau_n\}$ is a set
of consecutive roots and $r(\tau_n)\dots r(\tau_1)=c$, the set
$\{\tau_1,\dots,\tau_n\}$ is linearly independent by note $3.1$ of
\cite{Lattices}.

\begin{proposition}\label{prop}
For $\tau_i=\rho_{nh/2-n+i}$ and
\[\mathbf{v}=\tau_1+a\tau_2+a^2\tau_3+\dots+a^{n-1}\tau_n,\]
$\mu(\rho_i) \cdot \mathbf{v}
>0$ for all $1\leq i\leq nh/2$.
\end{proposition}

\begin{proof}
Recall from proposition $4.6$ of \cite{Lattices} that the
following properties hold.
\[\mu(\rho_i)\cdot\rho_j \geq 0 \mbox{ for } 1\leq i\leq
        j\leq nh/2.\]
\[\mu(\rho_{i+t})\cdot\rho_i=0 \mbox{ for } 1\leq t\leq n-1
        \mbox{ and for all } i.\]

Since $\tau_1,\dots,\tau_n$ are the last $n$ positive roots, it
follows that $\mu(\rho_i)\cdot \tau_j\geq 0$. Furthermore for each
$\rho_i$, there is at least one $\tau_j$ with
$\mu(\rho_i)\cdot\tau_j>0$ by linear independence of
$\{\tau_1,\dots,\tau_n\}$. Since all the coefficients of
$\mathbf{v}$ are strictly positive,
$\mu(\rho_i)\cdot\mathbf{v}>0$.
\end{proof}

\begin{proposition}
The projection of $\mu(X(c))$ onto the affine hyperplane
$H_\mathbf{v}$ where $\mathbf{v}$ is as in proposition 4.2 induces
an embedding of the homology of the non-crossing partition lattice
into the homology of the corresponding intersection lattice.
\end{proposition}

\begin{proof}
Recall from section \ref{homotopy} that homology generators for
the non-crossing partition lattice are identified with the
boundaries of facets of $X(c)$ and hence with facets of
$\mu(X(c))$. On the other hand, we can use the generic vector
$\mathbf{v}$ to identify homology generators of the intersection
lattice with cycles $g_R$ corresponding to regions $R$ such that
$R\cap H$ is nonempty and bounded. From \cite{PtoA}, the boundary
of each facet of $\mu(X(c))$ is a union of pieces of reflection
hyperplanes. It follows that vertices $\mu(\rho_i)$ for $1\leq
i\leq nh/2$ are rays and each facet of $\mu(X(c))$ projects to a
union of affine slices of the form $R\cap H$. Furthermore, the
projection of distinct $\mu(X(c))$ facets have disjoint
interiors. \\

We denote the projection map by $p:\mu(X(c))\rightarrow H$ and by
$p_*$ the induced map from the homology of the non-crossing
partition lattice to the homology of the intersection lattice.
Then $p_*$ takes the homology generator $g'_F$ corresponding to a
facet $F$ of $\mu(X(c))$ to the sum of the intersection lattice
homology generators $g_R$ corresponding to the affine slices
$R\cap H$ contained in $p(F)$. That is $p_*(g'_F)=\Sigma b_Rg_R$
where $b_R=1$ if $R\cap H$ is contained
in $p(F)$ and $0$ otherwise.\\

To establish injectivity of $p_*$, we observe that $p_*(\Sigma
a_Fg'_F)=\Sigma c_Rg_R$ where $c_R=0$ if $R$ is not contained in
$p(\mu(X(c)))$ and $c_R=a_F$ if $F$ is the unique facet satisfying
$R\subseteq p(F)$. Thus $\Sigma a_Fg'_F$ is an element of
$\mbox{Ker}(p_*)$ if and only if $a_F=0$ for all $F$.
\end{proof}

\begin{example}
For $W=C_3$ and for appropriate choices of fundamental domain and
simple system, the relevant regions are shown in figure \ref{fig}
where $i$ represents $\mu(\rho_i)$.\\

\begin{figure}[htp]
\centering
\includegraphics[height=0.4\textwidth]{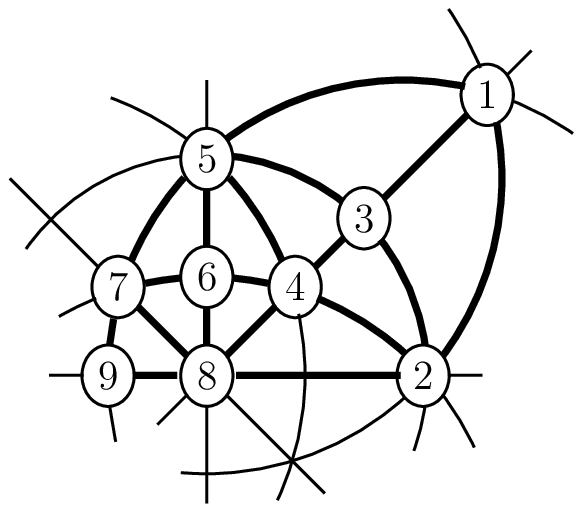}
\caption{}\label{fig}
\end{figure}

The basis for homology of the intersection lattice is formed by
cycles corresponding to regions in figure \ref{fig} which are
non-empty and bounded. For this
example, there are $15$ such regions. \\

Homology generators for the non-crossing partition lattice are
identified with the boundaries of facets of $\mu(X(c))$, of which
there are $10$ in this example. These facets are outlined in bold.
Note that the facet with corners $\mu(\rho_2), \mu(\rho_4),
\mu(\rho_8)$ is a union of two facets of the Coxeter complex and
therefore the embedding maps the homology element associated to
this facet to the sum of the two corresponding generators in the
homology of the intersection lattice.
\end{example}


\end{document}